\documentclass[11pt]{article}
\usepackage{amssymb,amsmath,mathrsfs,amsthm,indentfirst} 
\usepackage[numbers,sort&compress]{natbib}
\usepackage{graphicx}
\usepackage{subfigure}
\usepackage{diagbox}
\usepackage{setspace}
\allowdisplaybreaks

\catcode`!=11
\let\!int\int \def\int{\displaystyle\!int}
\let\!lim\lim \def\lim{\displaystyle\!lim}
\let\!sum\sum \def\sum{\displaystyle\!sum}
\let\!sup\sup \def\sup{\displaystyle\!sup}
\let\!inf\inf \def\inf{\displaystyle\!inf}
\let\!cap\cap \def\cap{\displaystyle\!cap}
\let\!max\max \def\max{\displaystyle\!max}
\let\!min\min \def\min{\displaystyle\!min}
\let\!frac\frac \def\frac{\displaystyle\!frac}
\catcode`!=12

\newtheorem{theorem}{\bf Theorem}[section]
\newtheorem{lemma}{\bf Lemma}[section]

\catcode`@=11 \@addtoreset{equation}{section} 
\catcode`@=12

\begin{document}

\title{
 Hopf bifurcation of a free boundary problem modeling tumor growth
with angiogenesis and two time delays}

\author{Haihua Zhou, Zejia Wang, Daming Yuan, Huijuan Song\thanks
{Corresponding author. E-mail: {\it songhj@jxnu.edu.cn}.} \quad
\\
{\normalsize School of Mathematics and Statistics,}
\\
{\normalsize Jiangxi Normal University, Nanchang
330022, P. R. China}
}

\date{}

\maketitle

\begin{abstract}
This paper concerns a free boundary problem modeling tumor growth with angiogenesis and two time delays. The two delays represent the time taken for cells to undergo mitosis and
modify the rate of cell loss because of apoptosis, respectively. We study the stability of stationary solutions and find that Hopf bifurcation occurs under some conditions, which extends the results of Xu. Furthermore, numerical simulations are performed to investigate the relationship among the rate of angiogenesis, two time delays and Hopf bifurcation.
\end{abstract}

\section{Introduction}

Over the past 50 years, a variety of free boundary problems of partial differential equations have been proposed to model the growth of solid tumors. Accordingly, asymptotic analysis, numerical simulations and rigorous mathematical analysis of such models have drawn great attention, and many interesting results have been established; see \cite{BC-95,BC-96,By-97,CX-07,FB-03,FL-15,FR-99,SZX-11,Pi-08,Xu-09,Xu-20,XZ-20,ZH-20} and references cited therein.

In this paper, we study the following free boundary problem which models the growth of a spherically symmetric tumor with angiogenesis and two delays:
\begin{align}
&\frac{1}{r^2}\frac{\partial}{\partial r}\left(r^2\frac{\partial\sigma}{\partial r}\right)=\Gamma\sigma, && 0<r<R(t), ~t>0,\label{xb-415}
\\
&\frac{\partial\sigma}{\partial r}(0,t)=0,\quad
\frac{\partial\sigma}{\partial r}(R(t),t)+\alpha[\sigma(R(t),t)-\sigma_\infty]=0,&& t>0, \label{xb-416}
\\
&\frac{d}{dt}\frac{4\pi R^3(t)}{3}=4\pi\int_0^{R(t-\tau_1)}\mu\sigma(r, t-\tau_1)r^2dr-4\pi\int_0^{R(t-\tau_2)}\mu\tilde{\sigma}r^2dr,
&& t>0,
\label{xb-417}
\\
&R(t)=\varphi(t), &&-\tau\le t\le 0, \label{xb-418}
\end{align}
where $R(t)$ denotes the unknown radius of the tumor at time $t$, $\sigma(r,t)$ is an unknown function describing the nutrient concentration within the tumor, $r$ represents the radial variable in $\mathbb{R}^3$, $\Gamma$ is the consumption rate coefficient of the nutrient by tumor cells, $\sigma_\infty$ is the concentration of nutrient in the host tissue of the tumor, $\alpha$ is a positive constant measuring the strength of the blood vessel system of the tumor: the smaller $\alpha$ is, the weaker the blood vessel system of the tumor will be; $\alpha=0$ corresponds to the case that the tumor does not have its own capillary vessel system so that nutrient is isolated on both sides of the surface of the tumor, and $\alpha=\infty$ indicates that the tumor is all surrounded by the blood vessels which reduces to the Dirichlet boundary condition
\begin{equation}
\label{eq(1.1)}
\sigma(R(t),t)=\sigma_\infty,\quad t>0.
\end{equation}

The equation \eqref{xb-417} is derived from conservation of mass, and the two terms on its right-hand side are explained as follows: the first term is the total volume increase in a unit time interval induced by cell proliferation; $\mu\sigma$ is the cell proliferation rate in unit volume.
The second term is total volume shrinkage in a unit time interval caused by cell apoptosis, or cell death due to aging; the rate of cell apoptosis is assumed to be a constant $\mu\tilde\sigma$ which is independent of $\sigma$. As in \cite{Pi-08,SZX-11,Xu-09},
here we denote by $\tau_1$ the time delay in cell proliferation, i.e., $\tau_1$ is the length of the period that a tumor cell undergoes a full process of mitosis,
and by $\tau_2$ the time taken for the cell to modify the rate of cell loss due to apoptosis. Finally, $\varphi(t)$ is a given positive function and $\tau=\max\{\tau_1, \tau_2\}$.

It was Byrne \cite{By-97} who initiated the idea of adding time delay on tumor growth models. Since then, interest in a study of the effects of time delay on the growth of tumors has arisen; see \cite{FB-03,CX-07,Pi-08,SZX-11,Xu-09,Xu-20,XZ-20,ZH-20} for example. In particular, except for significant biological background, equations with two time delays also have distinct physical background \cite{LRW-99,Pi-07}. Based on the Byrne-Chaplain inhibitor-free tumor model \cite{BC-95}, the model \eqref{xb-415}--\eqref{xb-418} under study introduces angiogenesis as in \cite{FL-15} and two independent time delays as in \cite{Pi-08,SZX-11,Xu-09}, respectively. If $\tau_2=0$, then the model \eqref{xb-415}--\eqref{xb-418} was studied in \cite{Xu-20,XZ-20}, where it was proved that the magnitude of the time delay in cell proliferation, i.e., $\tau_1$, does not affect the final dynamical behavior of solutions, but may have an effect on the rate of the evolvement process. While if $\tau_2>0$ and the Dirichlet boundary condition \eqref{eq(1.1)} is imposed, Xu \cite{Xu-09} found that being different from the case $\tau_2=0$ \cite{CX-07}, Hopf bifurcation may occur and the influence of time delays on the Hopf bifurcation was discussed; see \cite{SZX-11} for the model with an action of inhibitor.

Motivated by \cite{SZX-11,Xu-09}, we aim at investigating the stability of stationary solutions of the system \eqref{xb-415}--\eqref{xb-418}, as well as the influence of angiogenesis and two time delays on the Hopf bifurcation when one of delays is used as a bifurcation parameter. The analysis will be made in the framework of delay differential equations.

The structure of this paper is as follows. In Section 2, we establish the existence and uniqueness of transient solutions. In Section 3, we make analysis of the stability of stationary solutions and the existence of local Hopf bifurcation. In Section 4, we perform some numerical simulations by using Matlab for two purposes: one is to verify our theoretical results, the other is to observe the relationship among the rate of angiogenesis $\alpha$, two time delays $\tau_1$, $\tau_2$, and Hopf bifurcation.

\section{Existence and uniqueness of transient solutions}

It was proved in \cite{FL-15} that a unique solution of \eqref{xb-415}--\eqref{xb-416} is given explicitly by
\begin{equation}
\label{xb-431}
\sigma(r,t)=\frac{\alpha\sigma_{\infty} }{\alpha+\sqrt{\Gamma}g(\sqrt{\Gamma}R(t))} \frac{f(\sqrt{\Gamma}r)}{f(\sqrt{\Gamma}R(t))}\quad{\rm for}~0<r<R(t),~t>0,
\end{equation}
where
$$
f(x)=\frac{\sinh x}{x},\quad
g(x)=\frac{f'(x)}{f(x)}=\coth x-\frac{1}{x}.
$$
Setting
\begin{equation}
\label{eq(2.5)}
a_1=\mu\sigma_\infty,\quad\Lambda=\frac{\tilde\sigma}{3\sigma_\infty},\quad
\eta(t)=\sqrt{\Gamma}R(t),
\end{equation}
\begin{equation}
\label{eq(2.1)}
p(x)=\frac{g(x)}{x}=\frac{x\coth x-1}{x^2},\quad
l(x)=\frac{\alpha p(x)}{\alpha+\sqrt{\Gamma}g(x)}
\end{equation}
and plugging \eqref{xb-431} into \eqref{xb-417}, one obtains
\begin{align}
\label{xb-432}
\frac{d\eta(t)}{dt}&=a_1\eta(t)\left[\frac{\alpha p(\eta(t-\tau_1))}{\alpha+\sqrt{\Gamma}g(\eta(t-\tau_1))}\left(\frac{\eta(t-\tau_1)}{\eta(t)}\right)^3
-\Lambda\left(\frac{\eta(t-\tau_2)}{\eta(t)}\right)^3\right]\nonumber
\\
&=a_1 \eta(t)\left[l(\eta(t-\tau_1))\left(\frac{\eta(t-\tau_1)}{\eta(t)}\right)^3
-\Lambda\left(\frac{\eta(t-\tau_2)}{\eta(t)}\right)^3\right].
\end{align}
If we further denote
$$
\omega(t)=\eta^{3}(t),\quad
a=3a_1,
$$
then it follows from \eqref{xb-432} that
\begin{equation}
\label{xb-433}
\frac{d\omega(t)}{dt}=a\left[l\left(\omega^{1/3}(t-\tau_1)\right)\omega(t-\tau_1)
-\Lambda\omega(t-\tau_2)\right].
\end{equation}
Recalling that $\tau=\max\{\tau_1,\tau_2\}$, by the step method (see \cite{Ha-77}), if \eqref{xb-433} allows a solution for $t\in[-\tau, n\tau_3]$,
then the solution for $t\in[n\tau_3, (n+1)\tau_3]$, where $n\in\mathbb{N}$, $\tau_3=\min\{\tau_1, \tau_2\}$, must have the form
$$
\omega(t)=\omega(n\tau_3)+\int_{n\tau_3}^t a\left[l\left(\omega^{1/3}(s-\tau_1)\right)\omega(s-\tau_1)-\Lambda\omega(s-\tau_2)\right]ds.
$$
Since $s-\tau_1$, $s-\tau_2\in[-\tau, n\tau_3]$ for any $s\in[n\tau_3,(n+1)\tau_3]$, the step method guarantees the existence of a unique solution to the equation \eqref{xb-433} together with a positive initial function $\omega^0$.

The above discussion can be put together in the form of a theorem.

\begin{theorem}
\label{thm-2.1}
The problem \eqref{xb-415}--\eqref{xb-418} admits a unique solution $(\sigma(r,t),R(t))$ with $\sigma(r,t)$ defined by \eqref{xb-431} and $R(t)=\omega^{1/3}(t)/\sqrt{\Gamma}$, where $\omega(t)$ is the solution of the equation \eqref{xb-433} subject to the initial condition $\omega^0=(\sqrt{\Gamma}\varphi(t))^3$.
\end{theorem}

\section{Stability of stationary solutions and existence of local Hopf bifurcation}

We start with the analysis of stationary solutions of \eqref{xb-433}. Notice that the function $p(x)$ defined in \eqref{eq(2.1)} has the following properties: (see \cite{FL-15})
\begin{equation*}
p'(x)<0\quad{\rm for}~x>0,\qquad
\lim_{x\to0^+}p(x)=\frac{1}{3}, \qquad
\lim_{x\to\infty}p(x)=0,
\end{equation*}
which implies
\begin{equation}
\label{eq(2.4)}
l'(x)=\frac{\alpha^2p'(x)-\alpha\sqrt{\Gamma}p^2(x)}{(\alpha+\sqrt{\Gamma}g(x))^2}<0\quad{\rm for}~x>0
\end{equation}
and
\begin{equation}
\label{xb-4310}
\lim_{x\to0^+}l(x)=\frac{1}{3}, \quad \lim_{x\to\infty}l(x)=0.
\end{equation}
Thus, if $\sigma_\infty>\tilde\sigma$, then \eqref{xb-433} has a trivial stationary solution and a unique positive stationary solution $\omega_s$ satisfying
\begin{equation}
\label{xb-4311}
l\left(\omega_s^{1/3}\right)=\Lambda,
\end{equation}
and if $\sigma_\infty\le\tilde\sigma$, then the equation \eqref{xb-433} has only trivial stationary solution and no positive stationary solutions.

Before proceeding further, for
completeness and the reader's convenience, we state a stability criterion for
general delay differential equations; see \cite{SZX-11}.

\begin{lemma}
\label{lem-3.1}
Consider the equation
\begin{equation}
\label{eq(3.1)}
\frac{dx}{dt}=f(x(t-r_1),x(t-r_2)),
\end{equation}
with a nonnegative initial continuous function $x^0:[-r_0,0]\to(0,\infty)$, where $r_1$,
$r_2$ are the positive constants, $r_0=\max\{r_1,r_2\}$, and $f$ is a continuously
differentiable nonlinear function. Assume that \eqref{eq(3.1)} has a trivial stationary solution,
that is, $f(0,0)=0$. Let the linearized equation around the trivial solution of \eqref{eq(3.1)} be as follows:
\begin{equation}
\label{eq(3.2)}
\frac{dx}{dt}=-B_1x(t-r_1)-B_2x(t-r_2).
\end{equation}
Then

(1) if $B_1<0$, $B_2>|B_1|$ and $r_1\in(0,\pi/(2\sqrt{B_2^2-B_1^2})]$, then there exists $r_2^0>0$ such that for $r_2\in[0,r_2^0)$ the trivial solution to \eqref{eq(3.1)} is asymptotically stable and for $r_2=r_2^0$ the Hopf bifurcation occurs;

(2) if $B_1<0$, $0<B_2<|B_1|$, the trivial solution to \eqref{eq(3.1)} is unstable independently on the values of both delays, and there is no Hopf bifurcation.
\end{lemma}

We are now ready to give and prove the main results of this section, concerning the stability of stationary solutions and the existence of local Hopf bifurcation. We will take $\tau_2$ as a bifurcation parameter, for the biological meaning see \cite{Pi-08} and the references therein.
First, the stability of positive stationary solutions of the equation \eqref{xb-433} is investigated.

\begin{theorem}
\label{mainth-t434}
Let $\sigma_\infty>\tilde\sigma$ and $\tau_1\in(0,\pi/(2\sqrt{A_2^2-A_1^2})]$ with
\begin{equation}
\label{eq(2.2)}
A_1=-\frac{a}{3}[\omega_s^{1/3}l'(\omega_s^{1/3})+3l(\omega_s^{1/3})],\quad
A_2=a\Lambda.
\end{equation}
Then there exists $\tau_2^*>0$ such that for
$\tau_2\in[0, \tau_2^*)$ the positive stationary solution $\omega_s$ to \eqref{xb-433} is asymptotically stable, and for $\tau_2=\tau_2^*$ the Hopf bifurcation occurs.
\end{theorem}

\begin{proof}
Linearizing \eqref{xb-433} around the positive stationary solution $\omega_s$ and letting $v(t)=\omega(t)-\omega_s$, we get
\begin{equation*}
\frac{dv(t)}{dt}=-A_1v(t-\tau_1)-A_2v(t-\tau_2),
\end{equation*}
where $A_1$, $A_2$ are respectively defined in \eqref{eq(2.2)}.
We claim that $A_1<0$. Indeed, consider the function $H(y)=y^3l(y)$ and by \eqref{eq(2.1)},
$$
\frac{1}{H(y)}=\frac{1}{y^3p(y)}+\frac{\sqrt{\Gamma}}{\alpha y^2}.
$$
Since $(p(y)y^3)'>0$ for $y>0$ (see \cite{CX-07}), we have
$$
\left(\frac{1}{H(y)}\right)'<0\quad{\rm for}~y>0.
$$
Thus, $H'(y)>0$ for $y>0$, which gives $yl'(y)+3l(y)>0$ for any $y>0$. The fact that $A_1<0$ is therefore obtained. Furthermore, in view of \eqref{eq(2.4)} and \eqref{xb-4311},
$$
|A_1|=\frac{a}{3}[\omega_s^{1/3}l'(\omega_s^{1/3})+3 l(\omega_s^{1/3})]<al(\omega_s^{1/3})=a\Lambda=A_2.
$$
As a direct application of Lemma \ref{lem-3.1}, the desired result follows and the proof is complete.
\end{proof}

Next, we analyze the stability of the zero solution of \eqref{xb-433}.

\begin{theorem}
\label{thm-3.1}
(i) If $\sigma_\infty>\tilde\sigma$, then the trivial stationary solution to \eqref{xb-433} is unstable independently of the values of both delays, and there is no Hopf bifurcation.

(ii) If $\sigma_\infty<\tilde\sigma$ and $\tau_1\in(0,\pi/(2 \mu\sqrt{\tilde{\sigma}^2-\sigma^2_\infty})]$, then there exists $\tau_2^{**}>0$ such that
for $\tau_2\in[0, \tau_2^{**})$ the trivial stationary solution to \eqref{xb-433} is asymptotically stable, and for $\tau_2=\tau_2^{**}$ the Hopf bifurcation occurs.
\end{theorem}

\begin{proof}
Linearizing \eqref{xb-433} around the trivial stationary solution yields
\begin{equation*}
\frac{d\omega(t)}{dt}=-\left(-\frac{a}{3}\right)\omega(t-\tau_1)-a\Lambda\omega(t-\tau_2).
\end{equation*}
When $\sigma_\infty>\tilde\sigma$, we see $0<a\Lambda<a/3=|-a/3|$ by \eqref{eq(2.5)}. When $\sigma_\infty<\tilde\sigma$, there holds $|-a/3|=a/3<a\Lambda$. Using Lemma 3.1 again, we conclude the assertions (i) and (ii), and complete the proof of this theorem.
\end{proof}

\section{Computer simulations}

In this section, using Matlab, we will do some numerical simulations for the model \eqref{xb-415}--\eqref{xb-418}. First, we take the following parameter values:
\begin{equation}
\label{eq4.1}
\Gamma=1,\quad \mu=1, \quad \tilde{\sigma}=2,\quad \alpha=0.2, \quad\sigma_\infty=3.3, \quad \tau_1=0.025\frac{\pi}{2\sqrt{A_2^2-A_1^2}}, \quad\omega^0=0.01,
\end{equation}
where $A_1$, $A_2$ are given by \eqref{eq(2.2)}. Evidently, all conditions of Theorem \ref{mainth-t434} are satisfied.

\vskip12mm
\setlength{\unitlength}{1cm}
\begin{picture}(7,7.5)
\put(1.5,-0.5){\centering\includegraphics[height=80mm,width=120mm]{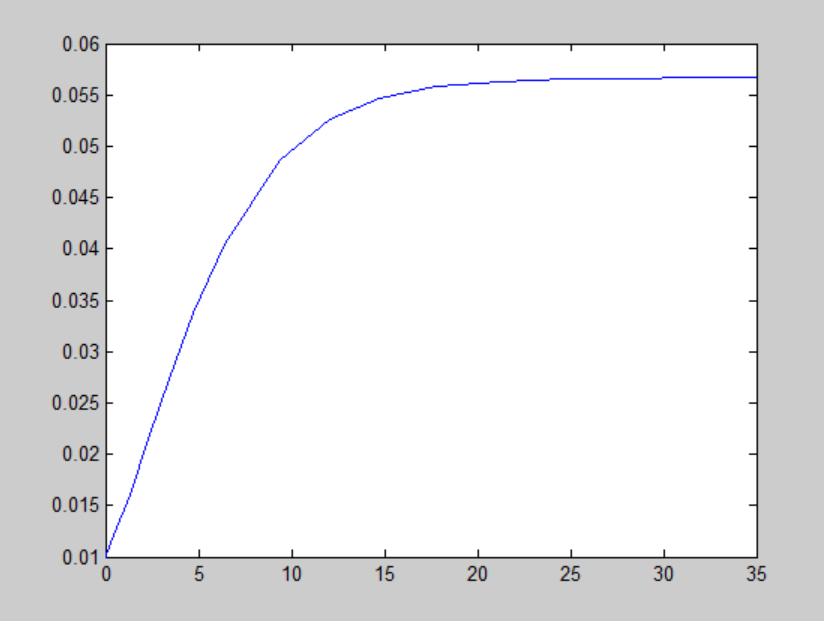}}
\put(2.2,-1){\small~Figure 1. An example of solution to equation \eqref{xb-433} for $\tau_2=0.0394$.}
\end{picture}
\setlength{\unitlength}{2cm}

\vskip12mm
\setlength{\unitlength}{1cm}
\begin{picture}(7,7.6)
\put(1.5,-0.5){\centering\includegraphics[height=80mm,width=120mm]{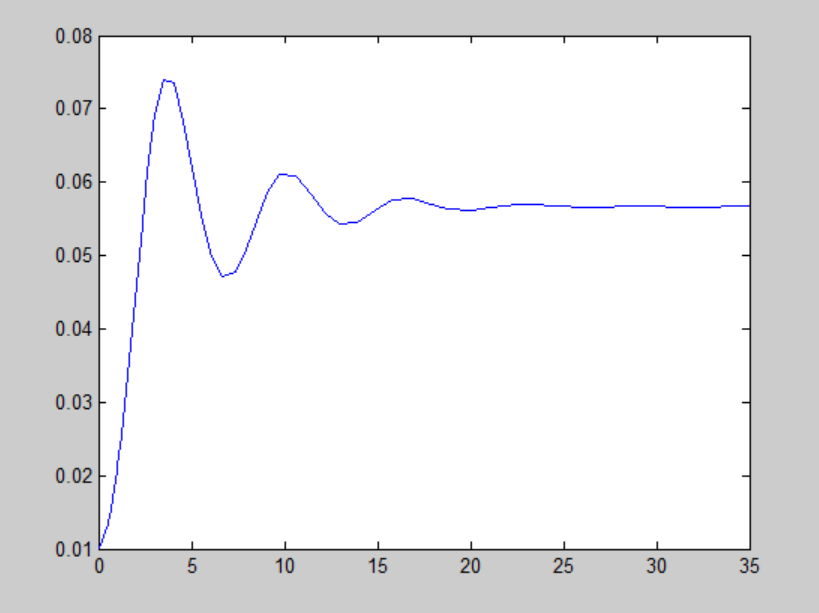}}
\put(2.2,-1){\small~Figure 2. An example of solution to equation \eqref{xb-433} for $\tau_2=0.5014$.}
\end{picture}
\setlength{\unitlength}{2cm}

\vskip12mm
\setlength{\unitlength}{1cm}
\begin{picture}(7,7.7)
\put(1.5,-0.5){\centering\includegraphics[height=80mm,width=120mm]{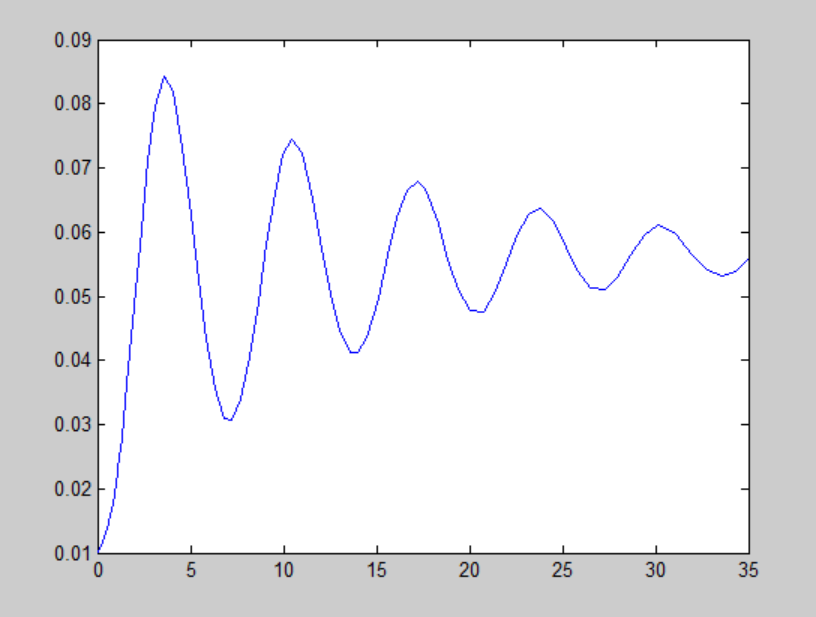}}
\put(2.2,-1){\small~Figure 3. An example of solution to equation \eqref{xb-433} for $\tau_2=0.5374$.}
\end{picture}
\setlength{\unitlength}{2cm}

\vskip12mm
\setlength{\unitlength}{1cm}
\begin{picture}(7,7.8)
\put(1.5,-0.5){\centering\includegraphics[height=80mm,width=120mm]{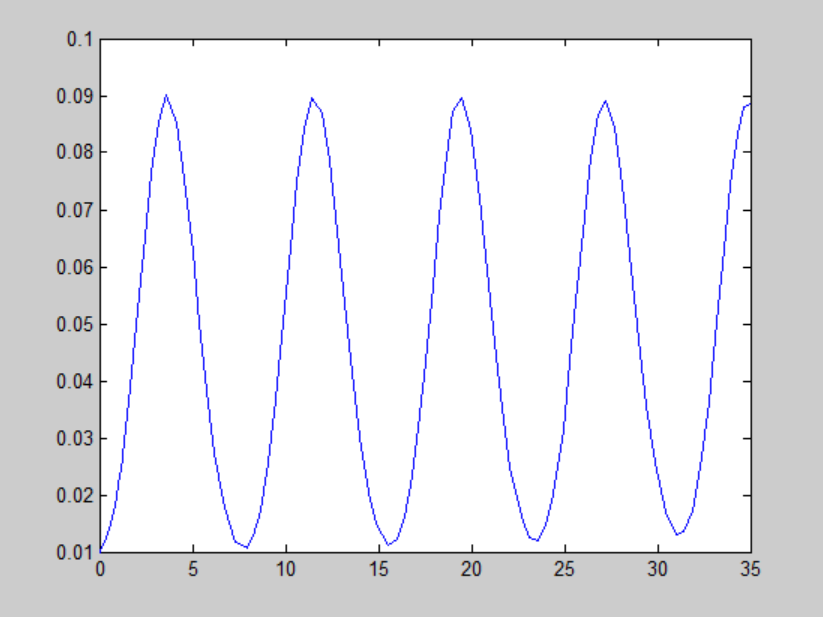}}
\put(2.2,-1){\small~Figure 4. An example of solution to equation \eqref{xb-433} for $\tau_2=0.5554$.}
\end{picture}
\setlength{\unitlength}{2cm}

\vskip12mm
\setlength{\unitlength}{1cm}
\begin{picture}(7,7.9)
\put(1.5,-0.5){\centering\includegraphics[height=80mm,width=120mm]{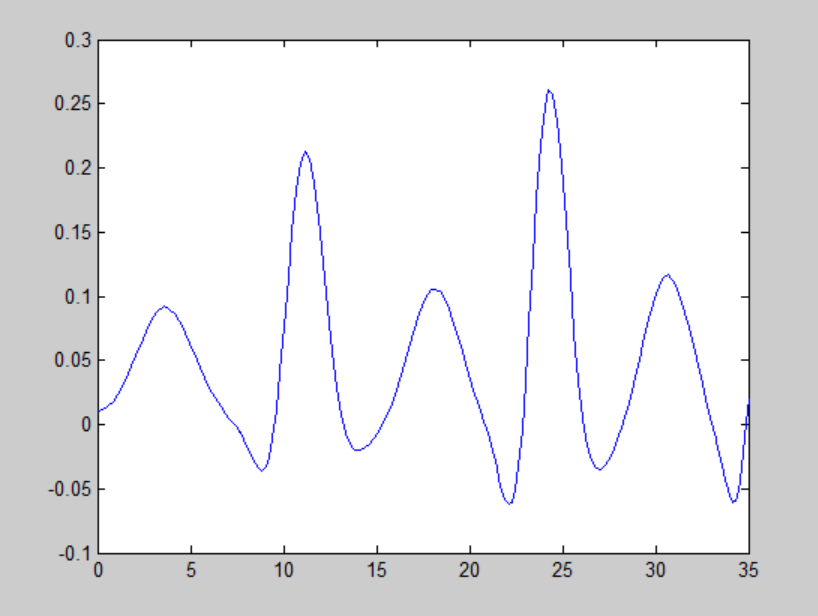}}
\put(2.2,-1){\small~Figure 5. An example of solution to equation \eqref{xb-433} for $\tau_2=0.5614$.}
\end{picture}
\setlength{\unitlength}{2cm}

\vskip12mm

In Figures 1--5, we present some examples of solutions to the equation \eqref{xb-433} for the same parameter values taken as \eqref{eq4.1} but different values of $\tau_2$. In this simulations, the Hopf bifurcation occurs at $\tau^*_2\approx 0.5556$. For small delays $\tau_2$, the solution tends to positive stationary state (see Figure 1). As $\tau_2$ increases, dumping oscillations arise (see Figures 2 and 3). When $\tau_2$ is very close to $\tau^*_2$, undumping oscillations are observed (see Figure 4). For some $\tau_2>\tau^*_2$, the solution does not remain positive (see Figure 5). These verify the results of Theorem \ref{mainth-t434}.

Next, taking the parameter values:
\begin{equation}
\label{eq4.2}
\Gamma=1,\quad \mu=1,\quad \tilde{\sigma}=2,\quad \sigma_\infty=3.3,\quad \omega^0=0.01,
\end{equation}
and different values of $\tau_1$, $\alpha$, respectively, we apply Matlab to obtain Tables \ref{tab1} and \ref{tab2}. Clearly, $A_2=2$ by \eqref{eq(2.2)}.
For $\alpha=0.2$, $1$, $2$, $3$, $4$, $5$, $6$, $7$, $8$, $9$, $10$, $20$, $30$, $40$, $50$, $60$, $70$, $80$, $90$, $100$, $1000$, $\infty$, where by $\alpha=\infty$, we mean the model \eqref{xb-415}--\eqref{xb-418} with the Robin boundary condition in \eqref{xb-416} replaced by the Dirichlet boundary condition \eqref{eq(1.1)}, using Matlab we find that there always holds $A_1\in[-1.8644, -1.5830]$; thus, $\pi/(2\sqrt{A_2^2-A_1^2})\in[1.2851, 2.1698]$. Hence, the conditions of Theorem \ref{mainth-t434} are fulfilled for every group of data.

\begin{table}[h]
 \centering\begin{tabular}{|c|c|c|c|c|c|c|c|c|c|c|c|}
\hline  \diagbox{$\tau_1$}{$\tau^*_2$}{$\alpha$} & 0.2 & 1 & 2 & 3 & 4 & 5 & 6 & 7 & 8 & 9 & 10 \\
\hline 0.05	& 0.5654 & 0.5665 & 0.5686 & 0.5696 & 0.5702 & 0.5706 & 0.5709 & 0.5711 & 0.5712 & 0.5714 & 0.5714 \\
\hline 0.1	& 0.6120 & 0.6128 & 0.6145 & 0.6153 & 0.6158 & 0.6161 & 0.6163 & 0.6164 & 0.6166 & 0.6166 & 0.6167 \\
\hline 0.15	& 0.6587 & 0.6593 & 0.6605 & 0.6611 & 0.6615 & 0.6617 & 0.6618 & 0.6619 & 0.6620 & 0.6621 & 0.6621 \\
\hline 0.2	& 0.7055 & 0.7059 & 0.7066 & 0.7070 & 0.7072 & 0.7074 & 0.7075 & 0.7075 & 0.7076 & 0.7076 & 0.7077 \\
\hline 0.25	& 0.7523 & 0.7525 & 0.7528 & 0.7530 & 0.7531 & 0.7531 & 0.7532 & 0.7532 & 0.7532 & 0.7533 & 0.7533 \\
\hline 0.3	& 0.7992 & 0.7992 & 0.7991 & 0.7990 & 0.7990 & 0.7990 & 0.7989 & 0.7989 & 0.7989 & 0.7989 & 0.7989 \\
\hline 0.35	& 0.8462 & 0.8459 & 0.8453 & 0.8451 & 0.8449 & 0.8448 & 0.8448 & 0.8447 & 0.8447 & 0.8446 & 0.8446 \\
\hline 0.4	& 0.8931 & 0.8926 & 0.8917 & 0.8912 & 0.8909 & 0.8907 & 0.8906 & 0.8905 & 0.8905 & 0.8904 & 0.8904 \\
\hline 0.5	& 0.9871 & 0.9862 & 0.9843 & 0.9834 & 0.9829 & 0.9826 & 0.9824 & 0.9822 & 0.9821 & 0.9820 & 0.9819 \\
\hline 0.6	& 1.0811 & 1.0797 & 1.0770 & 1.0757 & 1.0750 & 1.0745 & 1.0742 & 1.0739 & 1.0737 & 1.0736 & 1.0735 \\
\hline 0.7	& 1.1751 & 1.1733 & 1.1698 & 1.1680 & 1.1670 & 1.1664 & 1.1660 & 1.1656 & 1.1654 & 1.1652 & 1.1650 \\
\hline 0.8	& 1.2692 & 1.2669 & 1.2625 & 1.2603 & 1.2591 & 1.2583 & 1.2578 & 1.2574 & 1.2571 & 1.2569 & 1.2566 \\
\hline 0.9	& 1.3632 & 1.3604 & 1.3552 & 1.3527 & 1.3512 & 1.3502 & 1.3496 & 1.3491 & 1.3488 & 1.3485 & 1.3482 \\
\hline 1.0	& 1.4573 & 1.4541 & 1.4480 & 1.4450 & 1.4433 & 1.4422 & 1.4414 & 1.4408 & 1.4404 & 1.4401 & 1.4398 \\
\hline 1.1	& 1.5512 & 1.5476 & 1.5407 & 1.5373 & 1.5353 & 1.5340 & 1.5332 & 1.5325 & 1.5321 & 1.5317 & 1.5314 \\
\hline 1.2	& 1.6453 & 1.6412 & 1.6334 & 1.6296 & 1.6273 & 1.6259 & 1.6249 & 1.6242 & 1.6237 & 1.6232 & 1.6229 \\
\hline
\end{tabular}
\caption{Relations among $\alpha$, $\tau_1$ and $\tau_2^*$.}
\label{tab1}
\end{table}

\vskip6mm

\begin{table}[h]
 \centering\begin{tabular}{|c|c|c|c|c|c|c|c|c|c|c|c|}
\hline  \diagbox{$\tau_1$}{$\tau^*_2$}{$\alpha$}
        & 20 & 30 & 40 & 50 & 60 & 70 & 80 & 90 & 100 & 1000 & $\infty$ \\
\hline 0.05	& 0.5719 & 0.5720 & 0.5721 & 0.5721 & 0.5722 & 0.5722 & 0.5722 & 0.5722 & 0.5722 & 0.5723 & 0.5723 \\
\hline 0.1	& 0.6171 & 0.6172 & 0.6172 & 0.6173 & 0.6173 & 0.6173 & 0.6173 & 0.6173 & 0.6173 & 0.6174 & 0.6174 \\
\hline 0.15	& 0.6624 & 0.6625 & 0.6625 & 0.6625 & 0.6625 & 0.6626 & 0.6626 & 0.6626 & 0.6626 & 0.6626 & 0.6626 \\
\hline 0.2	& 0.7078 & 0.7079 & 0.7079 & 0.7079 & 0.7079 & 0.7079 & 0.7079 & 0.7079 & 0.7079 & 0.7080 & 0.7080 \\
\hline 0.25	& 0.7533 & 0.7533 & 0.7534 & 0.7534 & 0.7534 & 0.7534 & 0.7534 & 0.7534 & 0.7534 & 0.7534 & 0.7534 \\
\hline 0.3	& 0.7989 & 0.7989 & 0.7989 & 0.7989 & 0.7989 & 0.7989 & 0.7989 & 0.7989 & 0.7989 & 0.7989 & 0.7989 \\
\hline 0.35	& 0.8445 & 0.8445 & 0.8445 & 0.8444 & 0.8444 & 0.8444 & 0.8444 & 0.8444 & 0.8444 & 0.8444 & 0.8444 \\
\hline 0.4	& 0.8902 & 0.8901 & 0.8901 & 0.8900 & 0.8900 & 0.8900 & 0.8900 & 0.8900 & 0.8900 & 0.8900 & 0.8900 \\
\hline 0.5	& 0.9815 & 0.9814 & 0.9813 & 0.9813 & 0.9813 & 0.9813 & 0.9812 & 0.9812 & 0.9812 & 0.9812 & 0.9812 \\
\hline 0.6	& 1.0729 & 1.0727 & 1.0727 & 1.0726 & 1.0725 & 1.0725 & 1.0725 & 1.0725 & 1.0725 & 1.0724 & 1.0724 \\
\hline 0.7	& 1.1643 & 1.1641 & 1.1640 & 1.1639 & 1.1639 & 1.1638 & 1.1638 & 1.1638 & 1.1637 & 1.1636 & 1.1636 \\
\hline 0.8	& 1.2557 & 1.2555 & 1.2553 & 1.2552 & 1.2552 & 1.2551 & 1.2551 & 1.2550 & 1.2550 & 1.2549 & 1.2548 \\
\hline 0.9	& 1.3472 & 1.3468 & 1.3466 & 1.3465 & 1.3464 & 1.3464 & 1.3463 & 1.3463 & 1.3463 & 1.3461 & 1.3461 \\
\hline 1.0	& 1.4386 & 1.4381 & 1.4379 & 1.4378 & 1.4377 & 1.4377 & 1.4376 & 1.4376 & 1.4376 & 1.4373 & 1.4373 \\
\hline 1.1	& 1.5299 & 1.5295 & 1.5292 & 1.5291 & 1.5290 & 1.5289 & 1.5289 & 1.5288 & 1.5288 & 1.5285 & 1.5285 \\
\hline 1.2	& 1.6213 & 1.6208 & 1.6205 & 1.6203 & 1.6202 & 1.6201 & 1.6201 & 1.6200 & 1.6200 & 1.6197 & 1.6197 \\
        \hline
\end{tabular}
\caption{Relations among $\alpha$, $\tau_1$ and $\tau_2^*$.}
\label{tab2}
\end{table}

\vskip6mm

\setlength{\unitlength}{1cm}
\begin{picture}(7,7.9)
\put(1.5,-0.5){\centering\includegraphics[height=80mm,width=120mm]{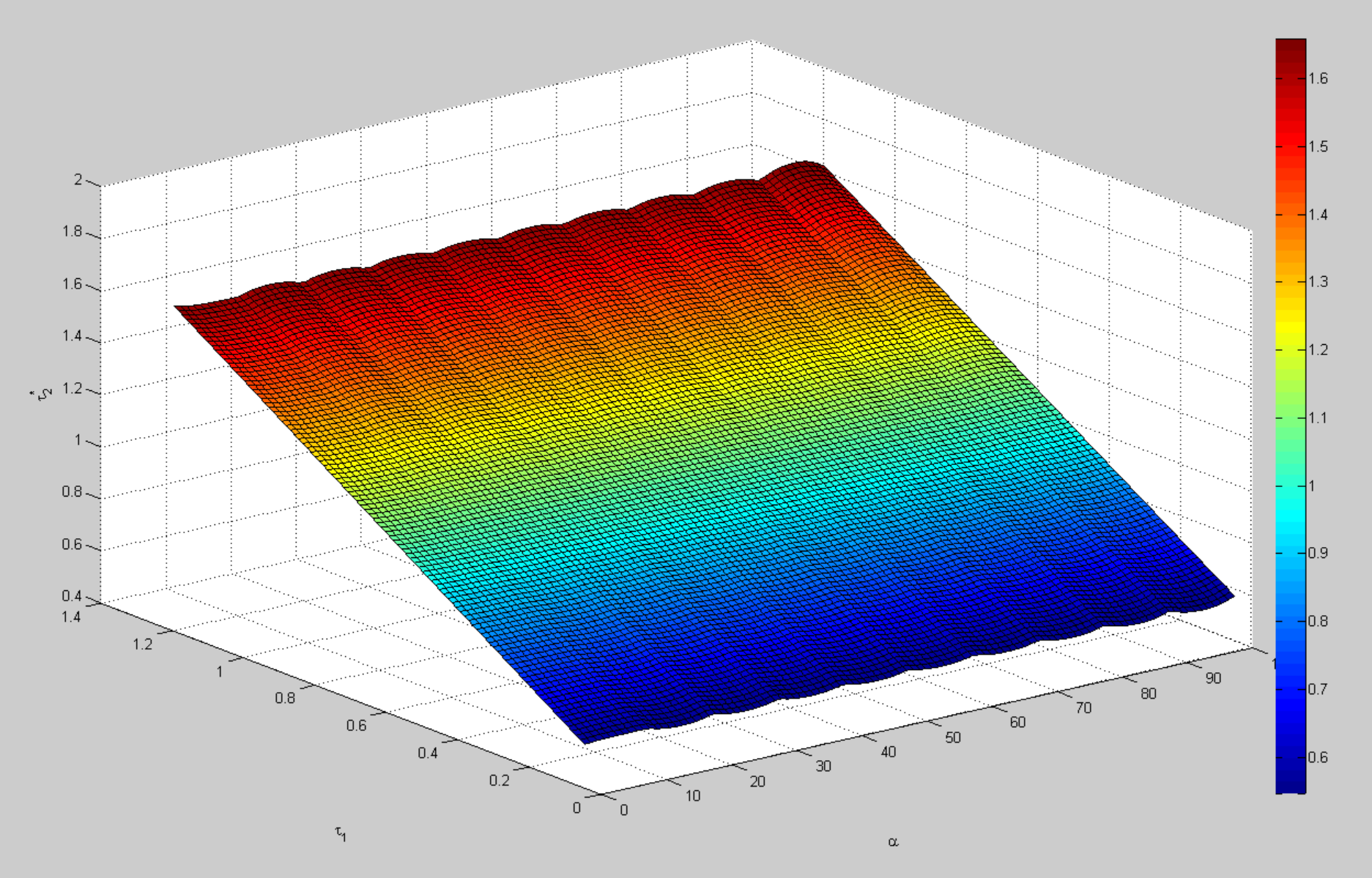}}
\put(3.5,-1){\small~Figure 6. The data of Tables \ref{tab1} and \ref{tab2} are drawn with Matlab. }
\end{picture}
\setlength{\unitlength}{2cm}
\vskip14mm

\newpage
As shown in Figure 6, generated by the data from Tables 1 and 2 with the help of Matlab, for fixed $\alpha$, $\tau^*_2$ is linearly increasing with respect to $\tau_1$; however, how $\tau^*_2$ varies as $\alpha$ varies is not obvious when $\tau_1$ is fixed. Nevertheless, one can see from the data, listed in Tables \ref{tab1} and \ref{tab2}, that when $\tau_1<0.3$, $\tau^*_2$ is nondecreasing with respect to $\alpha$, when $\tau_1\ge0.3$, $\tau^*_2$ is nonincreasing with respect to $\alpha$, and the value of the Hopf bifurcation point $\tau^*_2$ for $\alpha=1000$ is quite close to that for $\alpha=\infty$. In other words, the data in Tables \ref{tab1} and \ref{tab2} roughly reveal that the stability of steady state may be improved by increasing the time delay in proliferation $\tau_1$, or increasing the rate of angiogenesis if $\tau_1$ is small enough, or decreasing the rate of angiogenesis provided that $\tau_1$ is large. Finally, on the basis of a large angiogenesis intensity, increasing the rate of angiogenesis may have little effect on the Hopf bifurcation.

\section*{Acknowledgments}
This work was partly supported by the National Natural Science Foundation of China (No. 11861038, No. 11771156 and No. 12161045).

\end{document}